\documentclass[12pt]{article}
 \usepackage{amsmath}
\usepackage{amsfonts}
\usepackage{tikz}
\usepackage[title]{appendix}
\usepackage{subfig}
\usepackage{array}
\usepackage{blkarray}
\usepackage{tabu}
\usepackage{nth}
\usepackage[colorlinks=true, allcolors=black]{hyperref}

\usepackage[
  separate-uncertainty = true,
  multi-part-units = repeat
]{siunitx}

\usetikzlibrary{arrows}
\usepackage{epsfig,}
\usepackage{epsfig}
\tikzset{
    vertex/.style = {
        circle,
        draw,
        outer sep = 3pt,
        inner sep = 3pt,
    },edge/.style = {->,> = latex'}
}
\usepackage{amssymb}
\usepackage{enumitem}
\usepackage{amsthm}
\usepackage{dsfont}
\def\diag{\mathop{\rm diag}}
\def\adj{\mathop{\rm adj}}

\def\rank{\mathop{\rm rank}}
\newcommand{\rr}{\mathbb{R}}
\newcommand{\1}{\mathbf{1}}

\newcommand{\n}{\frac{n}{2}}

\def\det{{\rm det}}

\def\cir{{\rm Circ}}

\def\M{\widetilde{M}}
\def\L{\mathcal{L}}
\def\D{\widetilde{D}}

\newtheorem{theorem}{Theorem}

\newtheorem{lemma}{Lemma}

\begin{document}
\begin{center}
\begin{large}
On distance matrices of helm graphs obtained from wheel graphs with an even number of vertices
\end{large}
\end{center}
\begin{center}
Shivani Goel \\
\today
\end{center}

\begin{abstract}
Let $n \geq 4$. The helm graph $H_n$ on $2n-1$ vertices is obtained from the wheel graph $W_n$ by adjoining a pendant edge to each vertex of the outer cycle of $W_n$. Suppose $n$ is even. Let $D := [d_{ij}]$ be the distance matrix of $H_n$. In this paper, we first show that $\det(D) = 3(n-1)2^{n-1}.$ Next, we find a matrix $\L$ and a vector $u$ such that 
\[D^{-1} = -\frac{1}{2}\L+\frac{4}{3(n-1)}uu'.\]
We also prove an interlacing property between the eigenvalues of $\L$ and $D$. 
\end{abstract}

{\bf AMS CLASSIFICATION.} 05C50

\section{Introduction} \label{intro}
Let $G := (V,E)$ be a simple connected graph with vertex set $V:=\{1,2,\dotsc,n\}$ and edge set $E$. We denote a pair of adjacent vertices $i$ and $j$ in $G$ by $(i,j)$. There are three common matrices associated with $G$. The first one is the adjacency matrix $A$. The second matrix is the Laplacian matrix $L:=\diag(\delta_1,\dotsc,,\delta_n)- A$, where $\delta_i$ is the degree of vertex $i$. The third one is the distance matrix which is defined as follows. For $i \neq j$, let $d_{ij}$ denote the length of a shortest path connecting $i$ and $j$. For each $i \in V$, let $d_{ii}=0$. The matrix $D := [d_{ij}]$ is called the distance matrix of $G$. In this paper, we consider distance matrices. These matrices have wide literature and many applications. For a comprehensive introduction, see \cite{must}. 

Let $T$ be a tree with vertex set $V = \{1,2,\dotsc,n\}$. Suppose $D(T)$ and $L(T)$ denote the distance and Laplacian matrices of $T$, respectively. In Graham and Pollack \cite{GP}, the following elegant formula is obtained to compute the determinant of $D(T)$:
\begin{equation*}
    \det(D(T)) = (-1)^{n-1}(n-1) 2^{n-2}. 
\end{equation*}
In a subsequent paper by Graham and Lov\'asz \cite{Graham}, a remarkable formula is obtained to compute the inverse of $D(T)$, which says the following.
\begin{equation*}\label{tree}
    D(T)^{-1}=-\frac{1}{2}L(T) + \frac{1}{2(n-1)}\tau \tau',
\end{equation*} 
where $\tau = (2-\delta_1,\dotsc,2-\delta_n)'$. These results motivate to compute the determinant and the inverse of distance matrices of connected graphs other than trees. This problem is addressed for wheel graphs, complete graphs, complete bipartite graphs etc. in \cite{BALAJI2021274}, \cite{balaji2020distance},
\cite{sivasu} and \cite{hou}. Suppose $D$ is the distance matrix of helm graph obtained from a wheel graph with even number of vertices. 
In this paper, we show that $\det(D) = 3(n-1)2^{n-1}$
and \[D^{-1} = -\frac{1}{2}\L+\frac{4}{3(n-1)}uu'.\]

\subsection{Helm graphs}
Let $n \geq 4$. The helm graph $H_n$ on $2n-1$ vertices is obtained from the wheel graph $W_n$ by adjoining a pendant edge to each vertex of the outer cycle of $W_n$. An example of a helm graph obtained from the wheel graph on $6$ vertices is given in Figure \ref{fig_H6}.
\begin{figure}[!h]
\centering
\begin{tikzpicture}[shorten >=1pt, auto, node distance=3cm, ultra thick,
   node_style/.style={circle,draw=black,fill=white !20!,font=\sffamily\Large\bfseries},
   edge_style/.style={draw=black, ultra thick}]
\node[vertex] (1) at  (0,0) {$1$};
\node[vertex] (2) at  (2,-2) {$2$};
\node[vertex] (3) at  (2,0) {$3$}; 
\node[vertex] (4) at  (0,2) {$4$};  
\node[vertex] (5) at  (-2,0) {$5$};  
\node[vertex] (6) at  (-2,-2) {$6$};
\node[vertex] (7) at  (3.5,-3.5) {$7$};
\node[vertex] (8) at  (4,0) {$8$}; 
\node[vertex] (9) at  (0,4) {$9$};  
\node[vertex] (10) at  (-4,0) {$10$};  
\node[vertex] (11) at  (-3.5,-3.5) {$11$};
\draw  (1) to (2);
\draw  (1) to (3);
\draw  (1) to (4);
\draw  (1) to (5);
\draw  (1) to (6);
\draw  (2) to (3);
\draw  (3) to (4);
\draw  (4) to (5);
\draw  (5) to (6);
\draw  (6) to (2);
\draw  (2) to (7);
\draw  (3) to (8);
\draw  (4) to (9);
\draw  (5) to (10);
\draw  (6) to (11);
\end{tikzpicture}
\caption{Helm graph $H_6$} \label{fig_H6}
\end{figure}
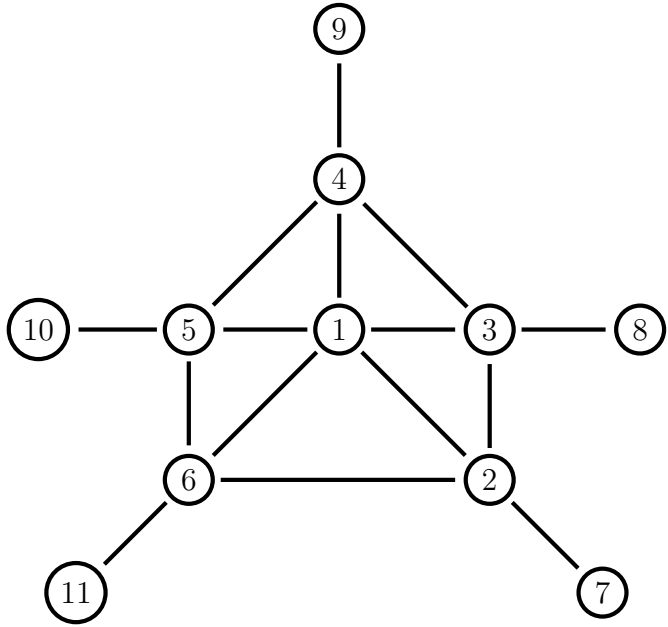
In \cite{Jakli}, it is shown that distance matrices of helm graphs are Euclidean distance matrices. Let $D$ be the distance matrix of $H_n$. From numerical computations, we observe that whenever $n$ is even, $D$ is an invertible matrix. In this paper, we first show that
$$\det(D) =  3(n-1)2^{n-1}.$$
Then we deduce an inverse formula for $D$. This formula says that 
\[D^{-1} = -\frac{1}{2}\L+\frac{4}{3(n-1)}uu',\]
where $u$ is non-zero and $\L$ is positive semidefinite, $\rank(\L) = n-1$ and row sums of $\L$ equal to zero.
We conclude the paper by proving an interlacing property between the eigenvalues of $\L$ and $D$. 

\section{Preliminaries}
In this section, we list a few notation, definitions, and preliminary results. 
\renewcommand{\labelenumi}{(P\arabic{enumi})}
\begin{enumerate}
    \item \label{(P1)} All vectors are considered as column vectors unless stated otherwise. 
    
    \item The transpose of a matrix $A$ is denoted by $A'$. We use $\adj(A)$ and $\det(A)$ to represent the adjoint and determinant of $A$, respectively.
    
    \item  The notation $n$ will stand for an even integer which is atleast four. We reserve the notation $m$ to denote the integer $2n-1$.
    
    \item \label{(P2)} The wheel graph on $n$ vertices is denoted by $W_n$. Let $V := \{1,\dotsc,n\}$ be the vertex set of $W_n$ and $E$ be the edge set of $W_n$. We label $W_n$ as follows: The center of $W_n$ is labeled $1$ and the $n-1$ vertices in the outer cycle of $W_n$ are labeled $2,3,\dotsc,n$, anticlockwise. Let $\widetilde{V} := V \cup \{n+1,n+2,\dotsc,m\}$ and $\widetilde{E} := E \cup \{(j,n+j-1):2\leq j \leq n\} $. Now the helm graph $H_n$ is the graph $(\widetilde{V}, \widetilde{E})$. It is clear that $H_n$ has $m$ vertices and $3(n-1)$ edges. For an example, see Figure \ref{fig_H6}.
    
    \item We use $\1$ to denote the column vector of all ones in $\rr^{n-1}$. The $({n-1}) \times ({n-1})$ matrix of all ones is denoted by $J$. The identity matrix of order ${n-1}$ is denoted by $I$. If $\mu \neq {n-1}$, then $\1_\mu$ denote the column vector of all ones in $\rr^\mu$ and $I_\mu$, $J_\mu$ will denote the $\mu \times \mu$ identity and all ones matrix, respectively.

    \item \label{(P3)} Let $s = (s_1,s_2,\dotsc,s_\mu)' \in \rr^\mu$. We use the notation $\cir(s')$ to denote the circulant matrix 
    \[ \left[\begin{array}{ccccccc}
        s_1 & s_2 & s_3 & \ldots & s_\mu \\
s_\mu & s_1 & s_2 & \ldots & s_{\mu-1} \\
s_{\mu-1} & s_\mu & s_1 &  \ldots &s_{\mu-2} \\
\vdots & \vdots & \vdots &  \ldots &\vdots \\
s_2 & s_3 & s_4 &  \ldots &s_1
    \end{array}\right].\]
    It is easy to note that if $C$ is a circulant matrix, then $\cir(s')C = \cir(s'C)$. 
    
\item \label{(P4)} Let $\D = \cir(v')$, where $v = (0,1,2,\dotsc,2,1)' \in \rr^{n-1}$. The distance matrix $D$ of the helm graph $H_n$ is now
\begin{equation}\label{D,D}
    D = \left[\begin{array}{ccc}
         0& \1'  & 2\1'\\
         \1& \D & \D+J\\
         2 \1& \D+J & \D+2(J-I)
    \end{array}\right].
\end{equation}
We note that $\D\1 = 2(n-3)\1$.
\item \label{(P5)} The following identity will be useful in the paper. If $n$ is even, then
$$\sum_{k=1}^{\frac{n}{2}-1} (-1)^k {(n-1-2k)}=\frac{2-n}{2}.$$

\item \label{(P6)} Let $A$ be a $\mu \times \mu$ matrix partitioned  
\[A = \left[\begin{array}{cc}
     A_{11}& A_{12} \\
     A_{21}&A_{22} 
\end{array}\right],\]
where $A_{11}$ and $A_{22}$ are square matrices. If $A_{11}$ is nonsingular then the Schur complement of $A_{11}$ in $A$ is  $A_{22}-A_{21}A_{11}^{-1}A_{12}$ and
\[\det(A) = \det(A_{11})\det(A_{22}-A_{21}A_{11}^{-1}A_{12}).\]
(See section 6.3 in \cite{zhang2011matrix}).
\item \label{(P7)} Let $A$ be an $\mu \times \mu$ matrix. If $u$ and $v$ belong to $\rr^\mu$, then

\[\det(A+uv') = \begin{cases}
    (1+v'A^{-1}u)\det(A) &\mbox{if}~A~\mbox{ is invertible} \\
    \det(A)+v'\adj(A)u &\mbox{if}~A~\mbox{is not invertible}.
    \end{cases}\]
 This result is well known as matrix determinant lemma.

\item \label{(P8)} An $n \times n$ non-negative symmetric matrix $A$ is called a Euclidean distance matrix, if all the diagonal entries are equal to zero and $x'Ax \leq 0$ for all $x \in \{\1\}^{\perp}.$ A Euclidean distance matrix has exactly one positive eigenvalue. For details, see chapter 3 in \cite{alfakih}.
\end{enumerate}

\subsection{Special Laplacian matrix for $H_n$}\label{laplacian}
Let $n$ be even. For each $k \in \{1,2,\dotsc,\frac{n}{2}-1\}$, define a vector ${c^k} := (c_{1}^{k},\dotsc,c_{n-1}^{k}) '$ in $\rr^{n-1}$ by
\begin{equation*}
    c^k_j := \begin{cases}
    1 & j=k+1~\mbox{or}~j=n-k \\ 
    0 &\mbox{otherwise}.
    \end{cases}
\end{equation*}
We now define the special Laplacian matrix $\L$ for $H_n$ as follows:
\begin{equation}\label{D,Lap}
    \L := \frac{1}{2}\left[\begin{array}{ccc}
        {n-1} & -\1' & ~~0\\
         -\1& (n+1) I & -2I\\
         0 & -2I & ~~{2}I
    \end{array}\right]+\sum_{k=1}^{\n-1}(-1)^{k}\frac{(n-1)-2k}{2}\left[\begin{array}{ccc}
        0 & 0 & 0\\
         0& C_k & 0\\
         0 & 0 & 0
    \end{array}\right],
\end{equation}
where $C_k := \cir({c^k}^{'})$. 

\section{Results}
We now prove our main results. 

\subsection{Determinant}
For distance matrices of $H_4$ and $H_6$, we now compute the determinant directly and give the formulas. 

\begin{enumerate}
    \item[$\bullet$] Consider $H_4$.  
The distance matrix $D$ of $H_4$ is given by
\begin{equation*} 
      D=\left[
{\begin{array}{rrrrrrr}
0 & 1 & 1 & 1 & 2 & 2 & 2 \\
1 & 0 & 1 & 1 & 1 & 2 & 2 \\
1 & 1 & 0 & 1 & 2 & 1 & 2 \\
1 & 1 & 1 & 0 & 2 & 2 & 1 \\
2 & 1 & 2 & 2 & 0 & 3 & 3 \\
2 & 2 & 1 & 2 & 3 & 0 & 3 \\
2 & 2 & 2 & 1 & 3 & 3 & 0\end{array}}
\right].
\end{equation*}
Let $n=4$. By direct computation, we have
\[\det(D) = 72 = 3(n-1)2^{n-1}.\]
\item[$\bullet$] 
Consider the graph $H_6$ given in Figure \ref{fig_H6}. 
The distance matrix $D$ of $H_6$ is given by
\begin{equation*} 
      D=\left[
{\begin{array}{rrrrrrrrrrrrrrrrrr}
0 & 1 & 1 & 1 & 1 & 1 & 2 & 2 & 2 & 2 & 2 \\
1 & 0 & 1 & 2 & 2 & 1 & 1 & 2 & 3 & 3 & 2 \\
1 & 1 & 0 & 1 & 2 & 2 & 2 & 1 & 2 & 3 & 3 \\
1 & 2 & 1 & 0 & 1 & 2 & 3 & 2 & 1 & 2 & 3 \\
1 & 2 & 2 & 1 & 0 & 1 & 3 & 3 & 2 & 1 & 2 \\
1 & 1 & 2 & 2 & 1 & 0 & 2 & 3 & 3 & 2 & 1 \\
2 & 1 & 2 & 3 & 3 & 2 & 0 & 3 & 4 & 4 & 3 \\
2 & 2 & 1 & 2 & 3 & 3 & 3 & 0 & 3 & 4 & 4 \\
2 & 3 & 2 & 1 & 2 & 3 & 4 & 3 & 0 & 3 & 4 \\
2 & 3 & 3 & 2 & 1 & 2 & 4 & 4 & 3 & 0 & 3 \\
2 & 2 & 3 & 3 & 2 & 1 & 3 & 4 & 4 & 3 & 0\end{array}}
\right].
\end{equation*}
Let $n=6$. We compute
\[\det(D) = 480 = 3(n-1)2^{n-1}.\]
\end{enumerate}
We assume $n \geq 8$ is even. 
For $1\leq k \leq \n-1$, define $q^k:={c^k}'\D$. We shall write $q^k=(q_{1}^k,\dotsc,q_{n-1}^k)$ and define $f:=(f_{1},\dotsc,f_{n-1})$. Now $f$ is the row vector
\[f=\sum_{k=1}^{\n-1} (-1)^{k} \frac{(n-1)-2k}{2}(q_{1}^k,\dotsc,q_{n-1}^{k}).\]
The next lemma gives a precise expression for $f$. 

\begin{lemma} \label{f-vector}
$f=(-1,\frac{3-n}{2},2-n,\dotsc,2-n,\frac{3-n}{2})$.
\end{lemma}
\begin{proof}
See Lemma 9 in \cite{BALAJI2021274}. 
\end{proof}

\begin{lemma}\label{L,matrixB}
Define $$B := \cir\big(\frac{n-1}{2}v'-\frac{3}{2}\1'+\sum_{k=1}^{\n-1}(-1)^{k}\frac{(n-1)-2k}{2}c_k'\D\big),$$ where $v$ is defined in (P\ref{(P4)}).  Then $B=-2I- \frac{1}{2} J.$
\end{lemma}
\begin{proof} We note that
 \[B=\cir\big(-\frac{3}{2}\1'+\frac{(n-1)}{2}v'+f\big).\]
 Since \(v=(0,1,2,\dotsc,2,1)'\) and 
 \(f=(-1,\frac{3-n}{2},2-n,\dotsc,2-n,\frac{3-n}{2}), \)
\begin{equation*}
    \begin{aligned}
        \frac{n-1}{2}v'-\frac{3}{2} \1'+f
        &=(-\frac{5}{2},-\frac{1}{2},\dotsc,-\frac{1}{2}).
    \end{aligned}
\end{equation*} 
By an easy manipulation we have,
\begin{equation*}
    B=\cir(-\frac{5}{2},-\frac{1}{2},\dotsc,-\frac{1}{2})=-2 I- \frac{1}{2} J.
\end{equation*}
The proof is complete.
\end{proof}

Let $W_n$ be a wheel graph on $n$ vertices. Suppose
\[M := \left[\begin{array}{ccc}
         0& \1' \\
         \1& \D 
    \end{array}\right],\]
    where $\widetilde{D}$ is defined in (P\ref{(P4)}). Then $M$ is the distance matrix of $W_n$. Also, by (\ref{D,D}), $M$ is a principal $n \times n$ submatrix of $D$.
We need the following result of \cite{BALAJI2021274}. 
\begin{theorem} \label{inversewheel}
Let $n \geq 4$ be an even integer and $M$ denote the distance matrix of the wheel graph $W_n$. Define $w \in \rr^n$ by
$w:=\frac{1}{4}(5-n,1,\dotsc,1)'$. Then,
\[M^{-1} = -\frac{1}{2} \widetilde{L}+\dfrac{4}{n-1}ww',\] where $\widetilde{L}$ is given by
\begin{equation}\label{laplacianwheel}
    \widetilde{L} := \dfrac{(n-1)}{2}I_n - \dfrac{1}{2} \left[{\begin{array}{cc}
        0 & \1'  \\ 
        \1 & 0
    \end{array}}\right] + \sum_{k=1}^{\frac{n}{2}-1} (-1)^k \dfrac{(n-1)-2k}{2} \left[{\begin{array}{cc}
        0 & 0  \\
        0 & C_k
    \end{array}}\right],
\end{equation}
and the matrices $C_k$ are defined in section \ref{laplacian}.
\end{theorem}
The following lemmas are needed in the sequel.
\begin{lemma}\label{Minverse}
Define $w \in \rr^n$ by
$w:=\frac{1}{4}(5-n,1,\dotsc,1)'$ and let $\widetilde{L}$ be given by (\ref{laplacianwheel}). Suppose $Z := \left[\begin{array}{cc}
     2 \1' \\ \D+J
\end{array}\right]$. Then the following are true.
\begin{enumerate}
\item[{\rm (i)}] $w'Z = \frac{n+3}{4}
\1'   .$
    \item[{\rm (ii)}] $\widetilde{L}Z= \left[\begin{array}{cc}
     \frac{5-n}{2}\1'  \\ 
     -2 I+ \frac{1}{2}J
\end{array}\right]. $
\item[{\rm (iii)}] $        M^{-1}Z =  \left[\begin{array}{cc}
     \frac{n-5}{4}\1'  \\ 
     I- \frac{1}{4}J
\end{array}\right] +\frac{n+3}{n-1}w
\1'
.$
\item[{\rm(iv)}] $        Z'M^{-1}Z = \D+\frac{2(n+1)}{n-1}J.$
\end{enumerate}
\end{lemma}
\begin{proof}
By a direct computation, we have
\begin{equation*}
    \begin{aligned}
   w'Z &= \frac{5-n}{2}
      \1'+ \frac{1}{4}\1'(\D+J).
    \end{aligned}
\end{equation*}
Since $\D\1 = 2(n-3)\1$, 
\begin{equation*}
    \begin{aligned}
   w'Z
&=
     \frac{5-n}{2}
      \1'+ \frac{n-3}{2}\1'+\frac{n-1}{4}\1'  \\
&=\frac{n+3}{4}
\1'.   
    \end{aligned}
\end{equation*}
The proof of (i) is complete. 

We now prove (ii). By (\ref{laplacianwheel}), 
\begin{equation*}
    \begin{aligned}
    \widetilde{L}Z &=     \left[\begin{array}{cc}
     X \\ Y
\end{array}\right],
    \end{aligned}
\end{equation*}
where \[X:=(n-1) \1'-\frac{1}{2}\1'(\D+J),\]
and 
\[Y:= \frac{(n-1)}{2}(\D+J)- J+\sum_{k=1}^{\frac{n}{2}-1} (-1)^k \frac{(n-1)-2k}{2} C_k(\D+J).\]
Using $\D\1 = 2(n-3)\1$, we deduce
\begin{equation*}
    \begin{aligned}
        X&=(n-1) \1'-(n-3)\1'-\frac{n-1}{2}\1'\\
        &= \frac{5-n}{2}\1'.
    \end{aligned}
\end{equation*}
By an easy verification,
\begin{equation*}
    \begin{aligned}
        Y = B+\frac{n}{2}J+\sum_{k=1}^{\frac{n}{2}-1} (-1)^k \frac{(n-1)-2k}{2}C_kJ,
    \end{aligned}
\end{equation*}
where $$B= \cir\big(\frac{n-1}{2}v'-\frac{3}{2}\1'+\sum_{k=1}^{\n-1}(-1)^{k}\frac{(n-1)-2k}{2}c_k'\D\big).$$ Since
$C_k \1 = 2\1$,
for each $k \in \{1,2,\dotsc,\n-1\}$, we have
\begin{equation}\label{matrixY}
    \begin{aligned}
        Y = B+\frac{n}{2}J+\sum_{k=1}^{\frac{n}{2}-1} (-1)^k (n-1-2k)J.
    \end{aligned}
\end{equation}
Using (P\ref{(P5)}) and Lemma \ref{L,matrixB} in (\ref{matrixY}), we have
\begin{equation*}
\begin{aligned}
    Y &= -2 I- \frac{1}{2} J+\frac{n}{2}J+\frac{2-n}{2}J \\
&= -2I+\frac{1}{2}J.
\end{aligned}
\end{equation*}
This completes the proof of (ii). 

To prove (iii), we note that by Theorem \ref{inversewheel}
\[M^{-1} = -\frac{1}{2} \widetilde{L}+\dfrac{4}{n-1}ww'.\]
Now, using the above inverse formula  along with (i) and (ii), the proof of (iii) is immediate. 

To prove (iv), we observe that
\begin{equation}\label{eq1}
    \begin{aligned}
Z'\left[\begin{array}{cc}
     \frac{n-5}{4}\1'  \\ 
     I- \frac{1}{4}J
\end{array}\right] &=   \frac{n-5}{2}J+(\D+J)(I- \frac{1}{4}J)\\ &=   \frac{n-5}{4}J+\D- \frac{1}{4}\D J.\\
    \end{aligned}
\end{equation}
Using $\D\1 = 2(n-3)\1$ in (\ref{eq1}), we have
\begin{equation}\label{eq3}
    \begin{aligned}
    Z'\left[\begin{array}{cc}
     \frac{n-5}{4}\1'  \\ 
     I- \frac{1}{4}J
\end{array}\right] &=   \frac{n-5}{4}J+\D- \frac{2(n-3)}{4}J \\ &= \D+\frac{1- n}{4}J,\\ 
    \end{aligned}
\end{equation}
Now, using (iii), we note that
\begin{equation}\label{eqn5}
    \begin{aligned}
        Z'M^{-1}Z = Z'\bigg(\left[\begin{array}{cc}
     \frac{n-5}{4}\1'  \\ 
     I- \frac{1}{4}J
\end{array}\right] +\frac{n+3}{n-1}w
\1'\bigg).
    \end{aligned}
\end{equation}
Using ($\ref{eq3}$) and (i) in (\ref{eqn5}), we get
\begin{equation*}
    \begin{aligned}
        Z'M^{-1}Z &= \D+\frac{1- n}{4}J +\frac{(n+3)^2}{4(n-1)}J\\
&=   \D+\frac{2(n+1)}{n-1}J.    \end{aligned}
\end{equation*}
This completes the proof of (iv).
\end{proof}

In our next result, we compute the Schur complement of $M$ in $D$. 
\begin{lemma}\label{Mschurcom}
Let $\M$ denote the Schur complement of $M$ in $D$. Then
\begin{enumerate}
    \item[{\rm (i)}] $\M =-2I-\frac{4}{n-1}J$.
    \item[{\rm (ii)}] $\det(\M) = (-3)2^{n-1}.$
\end{enumerate}        
\end{lemma}
\begin{proof}
By (P\ref{(P4)}), we have
\begin{equation*}
    D = \left[\begin{array}{ccc}
         0& \1'  & 2\1'\\
         \1& \D & \D+J\\
         2 \1& \D+J & \D+2(J-I)
    \end{array}\right].
\end{equation*}
Since $\M$ is the Schur complement of $M$ in $D$, by (P\ref{(P6)})
\begin{equation}\label{Mschur}
    \M = \D+2(J-I)- Z'M^{-1}Z,
\end{equation} 
where $Z = \left[\begin{array}{cc}
     2 \1' \\ \D+J
\end{array}\right]$.
Using item (iv) of Lemma \ref{Minverse} in (\ref{Mschur}), we have
\begin{equation*}
    \begin{aligned}
    \M = -2I-\frac{4}{n-1}J.
    \end{aligned}
\end{equation*}
The proof of (i) is complete. 

We now prove (ii). Using (P\ref{(P7)}) and (i), we note that
\begin{equation*}
    \begin{aligned}
        \det(\M) 
        &= \big(1-\frac{4}{n-1}\1'(-2I)^{-1}\1\big)\det{(-2I)}\\
         &= (-3)2^{n-1}.
    \end{aligned}
\end{equation*}
This completes the proof of (ii).
\end{proof}

\begin{theorem}\label{det}
If $D$ is the distance matrix of $H_n$, then 
\[\det(D) = 3(n-1)2^{n-1}.\]
\end{theorem}
\begin{proof}
Using (P\ref{(P6)}) and the fact that $\M$ is the Schur complement of $M$ in $D$, we have
\begin{equation}\label{dethelmdist}
    \det(D) = \det(M)\det(\M).
\end{equation}
Recall that $M$ is the distance matrix of $W_n$. By  Theorem 7 in \cite{zhang}, we have
\begin{equation}\label{detwheeldist}
    \det(M) = 1-n.
\end{equation}
By (\ref{dethelmdist}), (\ref{detwheeldist}), and Lemma \ref{Mschurcom}, we get 
\begin{equation*}
    \det(D) =  3(n-1)2^{n-1}.
\end{equation*}
The proof is complete.
\end{proof}

\subsection{Inverse Formula}

For distance matrices of $H_4$ and $H_6$, we now compute the inverse directly and give the formulas. 

\begin{enumerate}
    \item[$\bullet$] Consider $H_4$. Then $c^1=(0,1,1)'$.
The special Laplacian for $H_4$ can now be written easily using the definition:
\begin{equation*}
\begin{aligned}
    \L = \frac{1}{2}\left[\begin{array}{ccccccccccc}
~~3 & -1 & -1 & -1 & ~~0 & ~~0 & ~~0 \\
-1 & ~~5 & -1 & -1 & -2 & ~~0 & ~~0 \\
-1 & -1 & ~~5 & -1 & ~~0 & -2 & ~~0 \\
-1 & -1 & -1 & ~~5 & ~~0 & ~~0 & -2 \\
~~0 & -2 & ~~0 & ~~0 & ~~2 & ~~0 & ~~0 \\
~~0 & ~~0 & -2 & ~~0 & ~~0 & ~~2 & ~~0 \\
~~0 & ~~0 & ~~0 & -2 & ~~0 & ~~0 & ~~2\end{array}\right].
\end{aligned}
\end{equation*}
The distance matrix $D$ of $H_4$ is given by
\begin{equation*} 
      D=\left[
{\begin{array}{rrrrrrr}
0 & 1 & 1 & 1 & 2 & 2 & 2 \\
1 & 0 & 1 & 1 & 1 & 2 & 2 \\
1 & 1 & 0 & 1 & 2 & 1 & 2 \\
1 & 1 & 1 & 0 & 2 & 2 & 1 \\
2 & 1 & 2 & 2 & 0 & 3 & 3 \\
2 & 2 & 1 & 2 & 3 & 0 & 3 \\
2 & 2 & 2 & 1 & 3 & 3 & 0\end{array}}
\right].
\end{equation*}
By setting $u:=\frac{1}{4}(1,-1,-1,-1,2,2,2)'$, we note that 
\begin{equation} \label{h4}
-\frac{1}{2} \L+\frac{4}{3(n-1)}uu'= \frac{1}{18}\left[\begin{array}{cccccccccc}
-13 & ~~4 & ~~4 & ~~4 & ~~1 & ~~1 & ~~1 \\
~~4 & -22 & ~~5 & ~~5 & ~~8 & -1 & -1 \\
~~4 & ~~5 & -22 & ~~5 & -1 & ~~8 & -1 \\
~~4 & ~~5 & ~~5 & -22 & -1 & -1 & ~~8 \\
~~1 & ~~8 & -1 & -1 & -7 & ~~2 & ~~2 \\
~~1 & -1 & ~~8 & -1 & ~~2 & -7 & ~~2 \\
~~1 & -1 & -1 & ~~8 & ~~2 & ~~2 & -7\end{array}\right].
\end{equation}
The inverse of $D$ and the matrix in the right hand side of $(\ref{h4})$ are equal. This can be verified directly.

\item[$\bullet$] 
Consider the graph $H_6$ given in Figure \ref{fig_H6}. We see that 
\[c^1=(0,1,0,0,1)'~\mbox{and} ~c^2=(0,0,1,1,0)'.\]
The special Laplacian for $H_6$ can now be written easily using the definition:
\begin{equation*}
\begin{aligned}
    \L = \frac{1}{2}\left[\begin{array}{ccccccccccc}
~~5 & -1 & -1 & -1 & -1 & -1 & ~~0 & ~~0 & ~~0 & ~~0 & ~~0 \\
-1 & ~~7 & -3 & ~~1 & ~~1 & -3 & -2 & ~~0 & ~~0 & ~~0 & ~~0 \\
-1 & -3 & ~~7 & -3 & ~~1 & ~~1 & ~~0 & -2 & ~~0 & ~~0 & ~~0 \\
-1 & ~~1 & -3 & ~~7 & -3 & ~~1 & ~~0 & ~~0 & -2 & ~~0 & ~~0 \\
-1 & ~~1 & ~~1 & -3 & ~~7 & -3 & ~~0 & ~~0 & ~~0 & -2 & ~~0 \\
-1 & -3 & ~~1 & ~~1 & -3 & ~~7 & ~~0 & ~~0 & ~~0 & ~~0 & -2 \\
~~0 & -2 & ~~0 & ~~0 & ~~0 & ~~0 & ~~2 & ~~0 & ~~0 & ~~0 & ~~0 \\
~~0 & ~~0 & -2 & ~~0 & ~~0 & ~~0 & ~~0 & ~~2 & ~~0 & ~~0 & ~~0 \\
~~0 & ~~0 & ~~0 & -2 & ~~0 & ~~0 & ~~0 & ~~0 & ~~2 & ~~0 & ~~0 \\
~~0 & ~~0 & ~~0 & ~~0 & -2 & ~~0 & ~~0 & ~~0 & ~~0 & ~~2 & ~~0 \\
~~0 & ~~0 & ~~0 & ~~0 & ~~0 & -2 & ~~0 & ~~0 & ~~0 & ~~0 & ~~2\end{array}\right].
\end{aligned}
\end{equation*}
The distance matrix $D$ of $H_6$ is given by
\begin{equation*} 
      D=\left[
{\begin{array}{rrrrrrrrrrrrrrrrrr}
0 & 1 & 1 & 1 & 1 & 1 & 2 & 2 & 2 & 2 & 2 \\
1 & 0 & 1 & 2 & 2 & 1 & 1 & 2 & 3 & 3 & 2 \\
1 & 1 & 0 & 1 & 2 & 2 & 2 & 1 & 2 & 3 & 3 \\
1 & 2 & 1 & 0 & 1 & 2 & 3 & 2 & 1 & 2 & 3 \\
1 & 2 & 2 & 1 & 0 & 1 & 3 & 3 & 2 & 1 & 2 \\
1 & 1 & 2 & 2 & 1 & 0 & 2 & 3 & 3 & 2 & 1 \\
2 & 1 & 2 & 3 & 3 & 2 & 0 & 3 & 4 & 4 & 3 \\
2 & 2 & 1 & 2 & 3 & 3 & 3 & 0 & 3 & 4 & 4 \\
2 & 3 & 2 & 1 & 2 & 3 & 4 & 3 & 0 & 3 & 4 \\
2 & 3 & 3 & 2 & 1 & 2 & 4 & 4 & 3 & 0 & 3 \\
2 & 2 & 3 & 3 & 2 & 1 & 3 & 4 & 4 & 3 & 0\end{array}}
\right].
\end{equation*}
Setting $u:=\frac{1}{4}(-1,-1,-1-1,-1,-1,2,2,2,2,2)'$, the matrix $-\frac{1}{2} \L+\frac{4}{3(n-1)}uu'$ is given by
\begin{equation*} \label{w5}
\frac{1}{30}\left[\begin{array}{ccccccccccccccc}
-37 & ~~8 & ~~8 & ~~8 & ~~8 & ~~8 & -1 & -1 & -1 & -1 & -1 \\
~~8 & -52 & ~~23 & -7 & -7 & ~~23 & ~~14 & -1 & -1 & -1 & -1 \\
~~8 & ~~23 & -52 & ~~23 & -7 & -7 & -1 & ~~14 & -1 & -1 & -1 \\
~~8 & -7 & ~~23 & -52 & ~~23 & -7 & -1 & -1 & ~~14 & -1 & -1 \\
~~8 & -7 & -7 & ~~23 & -52 & ~~23 & -1 & -1 & -1 & ~~14 & -1 \\
~~8 & ~~23 & -7 & -7 & ~~23 & -52 & -1 & -1 & -1 & -1 & ~~14 \\
-1 & ~~14 & -1 & -1 & -1 & -1 & -13 & ~~2 & ~~2 & ~~2 & ~~2 \\
-1 & -1 & ~~14 & -1 & -1 & -1 & ~~2 & -13 & ~~2 & ~~2 & ~~2 \\
-1 & -1 & -1 & ~~14 & -1 & -1 & ~~2 & ~~2 & -13 & ~~2 & ~~2 \\
-1 & -1 & -1 & -1 & ~~14 & -1 & ~~2 & ~~2 & ~~2 & -13 & ~~2 \\
-1 & -1 & -1 & -1 & -1 & ~~14 & ~~2 & ~~2 & ~~2 & ~~2 & -13\end{array}\right].
\end{equation*}
It can be easily verified that the inverse of $D$ and the above matrix are equal. 
\end{enumerate}

In the rest of the paper, we assume $n \geq 8$ is even. 
The following result gives a simple expression for the matrix $\L D$.
\begin{lemma}\label{T,LD}
If $\L$ is the special Laplacian matrix defined in $(\ref{D,Lap})$ and $D$ is the distance matrix of $H_n$, then
\[\L D = \left[\begin{array}{ccc}
        \frac{1-n}{2} & \frac{5-n}{2}\1' & \frac{5-n}{2}\1'\\
        \\
         -\frac{1}{2}\1& B& 2I+B\\
         \\
         \1 & J &J-2I
    \end{array}\right],\]
    where $$B = \cir\big(\frac{n-1}{2}v'-\frac{3}{2}\1'+\sum_{k=1}^{\n-1}(-1)^{k}\frac{(n-1)-2k}{2}c_k'\D\big).$$
\end{lemma}
\begin{proof}
Using $(\ref{D,D})$ and $(\ref{D,Lap})$, we compute
\begin{equation}\label{matrixLD}
    \begin{aligned}
        \L D &=  \left[\begin{array}{ccc}
        \frac{1-n}{2} & A & A\\
         E& F& G\\
         \1 & J &J-2I
    \end{array}\right],
    \end{aligned}
\end{equation}
where 
\[A := \frac{n-1}{2}\1'-\frac{1}{2}\1'\D,\]
\[E := \frac{n-3}{2}\1+\sum_{k=1}^{\n-1}(-1)^{k}\frac{n-1-2k}{2}C_k\1,\]
\[F := -\frac{3}{2}J+\frac{n-1}{2}\D+\sum_{k=1}^{\n-1}(-1)^{k}\frac{n-1-2k}{2} C_k\D,\]
and 
\[G := 2I+\frac{n-5}{2}J+\frac{n-1}{2}\D+\sum_{k=1}^{\n-1}(-1)^{k}\frac{n-1-2k}{2}(C_k\D+2J).\]
As $\D\1 = 2(n-3)\1$, we get
\begin{equation*}\label{matrixA}
    A = \frac{5-n}{2} \1'.
\end{equation*}
Each $C_k$ is a circulant matrix specified by a vector in $\rr^{n-1}$ with exactly two ones and remaining entries equal to zero. Therefore, $C_k \1 = 2 \1$, for each $k \in \{1,2,\dotsc,\n-1\}$. Thus 
\begin{equation*}
    \begin{aligned}
        E &= \frac{n-3}{2}\1+\sum_{k=1}^{\n-1}(-1)^{k}{(n-1-2k)}\1. \\
    \end{aligned}
\end{equation*}
Now, using the identity in (P\ref{(P5)}), we get
\begin{equation*}\label{matrixE}
    \begin{aligned}
        E &= \frac{n-3}{2}\1+\frac{2-n}{2}\1 = -\frac{1}{2}\1. \\
    \end{aligned}
\end{equation*}
Since $\D=\cir (v' )$, $\1 \1'=\cir(\1')$ and $C_k\D=\cir({c^k}^{'} \D)$, we get
\begin{equation*}\label{matrixF}
    F = \cir\big(\frac{n-1}{2}v'-\frac{3}{2}\1'+\sum_{k=1}^{\n-1}(-1)^{k}\frac{(n-1)-2k}{2}c_k'\D\big) = B,
\end{equation*}
and 
\begin{equation*}
    \begin{aligned}
        G = 2I+\frac{n-5}{2}J+\sum_{k=1}^{\n-1}(-1)^{k}{(n-1-2k)}J+\cir\big(\frac{n-1}{2}v'+\sum_{k=1}^{\n-1}(-1)^{k}\frac{(n-1)-2k}{2}c_k'\D\big).
    \end{aligned}
\end{equation*}
Using (P\ref{(P5)}), we have
\begin{equation*}
    \begin{aligned}
        G &= 2I+\big(\frac{n-5}{2}+\frac{2-n}{2}\big)J+\cir\big(\frac{n-1}{2}v'+\sum_{k=1}^{\n-1}(-1)^{k}\frac{(n-1)-2k}{2}c_k'\D\big) \\
        &= 2I-\frac{3}{2}J+\cir\big(\frac{n-1}{2}v'+\sum_{k=1}^{\n-1}(-1)^{k}\frac{(n-1)-2k}{2}c_k'\D\big).
    \end{aligned}
\end{equation*}
Since $J = \cir(\1')$, we get
\begin{equation*}\label{matrixG}
    \begin{aligned}
        G &= 2I+\cir\big(\frac{n-1}{2}v'-\frac{3}{2}\1'+\sum_{k=1}^{\n-1}(-1)^{k}\frac{(n-1)-2k}{2}c_k'\D\big) \\
        &= 2I+B.
    \end{aligned}
\end{equation*}
The proof is complete by substituting the values of $A$, $E$, $F$ and $G$ in (\ref{matrixLD}).
\end{proof}

We now state and prove our main theorem.
\begin{theorem} $\label{inverse}$
Let $n \geq 4$ be even. If $D$ and $\L$ are the distance  and special Laplacian matrices of $H_n$, respectively, then
\[D^{-1} = -\frac{1}{2}\L+\frac{4}{3(n-1)}uu'.\]
where $u:=\frac{1}{4}(5-n,-1,\dotsc,-1,2,\dotsc,2)' \in \rr^{m}$ and $\L$ is defined in section \ref{laplacian}.
\end{theorem}
\begin{proof}
The case when $n=4$ and $n=6$ are already discussed. We assume $n \geq 8$. In view of Lemma \ref{L,matrixB} and \ref{T,LD}, 
\[\L D=
\left[\begin{array}{ccc}
        -\frac{n-1}{2} & \frac{5-n}{2}\1' & \frac{5-n}{2}\1'\\
        \\
         -\frac{1}{2}\1& B& 2I+B\\
         \\
         ~~~~~\1 & J &J-2I
    \end{array}\right],\]
where
 \[B=-2 I- \frac{1}{2} J.\]
By setting $u=\frac{1}{4}(5-n,-1,\dotsc,-1,2,\dotsc,2)'$, we deduce that
\begin{equation} \label{ld2}
\begin{aligned}
    \L D+2I_{m} &= 2u\1_{m}'.\\ 
\end{aligned}
\end{equation}
Next, we compute 
\begin{equation*} 
\begin{aligned}
Du&=\frac{1}{4}\left[\begin{array}{ccc}
         3\1'\1\\
         (5-n)\1+\D\1+2J\1\\
         2(5-n)\1+\D\1+3J\1-4\1
    \end{array}\right].    
\end{aligned}
\end{equation*}
Since $\D\1 = 2(n-3)\1$, we get
\begin{equation} \label{Du}
\begin{aligned}
Du&=\frac{3(n-1)}{4}\1_{m}.
\end{aligned}
\end{equation}
By (\ref{ld2}) and (\ref{Du}), we have
\begin{equation*}
    \begin{aligned}
        \big(-\frac{1}{2}L+\frac{4}{3(n-1)}uu'\big)D&=-\frac{1}{2}LD+\frac{4}{3(n-1)}uu'D\\ &=-\frac{1}{2}(2u\1_{m}'-2I_{m})+u\1_{m}' \\&=I_m.
    \end{aligned}
\end{equation*}
Hence
\[D^{-1} = -\frac{1}{2}\L+\frac{4}{3(n-1)}uu'.\]
The proof is complete.
\end{proof}

\subsection{Properties of the special Laplacian}
In this section, we show that $\L$ is a positive semidefinite matrix, has rank $m-1$ and all row sums equal to zero. We also show that all the cofactors of $\mathcal{L}$ are equal to $2^{n-3}$. The proofs of all the results in this section follows from the same technique used in \cite{BALAJI2021274}. We include the proofs here for completion.
\begin{theorem}
Row and column sums of $\L$ are zero and $\rank(\L)=m-1$.  
\end{theorem}
\begin{proof}
From equation $(\ref{ld2})$, we have
\begin{equation}\label{ld2copy} 
\L D +2 I_{m}=2 u\1_{m}'.
\end{equation}
Since $u=\frac{1}{4}(5-n,-1,\dotsc,-1,2,\dotsc,2)' \in \rr^{m}$, we have
 \[\1_{m}'u=\frac{5-n}{4}-\frac{n-1}{4}+\frac{2(n-1)}{4} = 1.\]
Thus, equation (\ref{ld2copy}) gives $\1' \L D=0$.
Let $p \in \rr^m$ be a non-zero vector such that $p' \L D=0$.
In view of (\ref{ld2copy}),  we have
\[p'(-2I_{m}+2u\1_{m}') = 0.\]
This gives
\[p'=(p'u)\1_{m}'.\]
So, $p$ is a multiple of $\1_{m}$. Thus,
nullity of $\L D$ is one and hence $\rank(\L D)=m-1$.
As $D$ is non-singular, $\rank(\L)=m-1$.
Since $\L$ is symmetric and $\1_{m}' \L D=0$ if and only if $\1_{m}' \L=0$, all the row and column sums of $\L$ are zero. The proof is complete.
\end{proof}

\begin{theorem}
$\L$ is positive semidefinite.
\end{theorem}
\begin{proof}
Since $\rank(\L)=m-1$ and $\L \1_{m}=0$, $\L \L^{\dag}$ is a symmetric idempotent matrix with null space equal to $\mbox{span}\{\1_{m}\}$.
Thus, $\L \L^{\dag}=I_{m}-\frac{J_{m}}{n}$. Define $P:=I_{m}-\frac{J_{m}}{n}$. By the identity $$\L D=-2I_{m}+2u \1_{m}',$$ we get 
$PDP=-2\L^{\dag}$. Let $D:=[d_{ij}]$ and $\L^{\dag}:=[a_{ij}]$. It is now easy to get the relation
\[d_{ij}=a_{ii}+a_{jj}-2a_{ij}. \]
From the above equation, 
\begin{equation} \label{ldj}
D=\diag(\L^{\dag})J_{m}+ J_{m}\diag( \L^{\dag}) -2 \L^{\dag}. 
\end{equation}
By Theorem 14 in \cite{Jakli}, $x' D x \leq 0$ for all $x \in \{\1_{m}\}^{\perp}$.
Now, $(\ref{ldj})$ implies that $x' \L^{\dag} x \geq 0$ for all $x \in \{\1_{m}\}^{\perp}$. We know that $\rank(\L)=m-1$ and $\L\1_{m}=0$. By the properties of Moore-Penrose inverse, we deduce that $x' \L x \geq 0$ for all $x \in \{\1_{m}\}^{\perp}$. Furthermore, since $\L \1_{m}=0$, it follows that $\L$ is positive semidefinite. The proof is complete.
\end{proof}

\begin{theorem}
All cofactors of $\L$ are equal to $2^{n-3}$.
\end{theorem}
\begin{proof}
Since $\L$ is symmetric and $\L \1=0$, all cofactors of $\L$ are equal. Let the common cofactor of $\L$ be $\delta$.
By the inverse formula,
\[D^{-1}=-\frac{1}{2} \L + \frac{4}{3(n-1)}uu'. \]
By using (P\ref{(P7)}) and $\1'_{m}u=1$, we get 
\begin{equation*}
\begin{aligned}
\det( D^{-1})&= \det (-\frac{1}{2} \L) + \frac{4}{3(n-1)} u' \mbox{adj} (-\frac{1}{2} \L)u \\
&=\frac{4}{3(n-1)} (-1)^{m-1} \frac{1}{2^{m-1}} \delta.
\end{aligned}
\end{equation*}
Since $m=2n-1$ and $n$ is even, we have
\begin{equation}\label{detDinv1}
\begin{aligned}
\det(D^{-1})=\frac{4}{3(n-1)} \frac{1}{2^{2n-2}} \delta. 
\end{aligned}
\end{equation}
By Theorem \ref{det}, 
\begin{equation}\label{detdinv2}
    \det(D^{-1})=\frac{1}{3(n-1)}\frac{1}{2^{n-1}}.
\end{equation}  Comparing (\ref{detDinv1}) and (\ref{detdinv2}), we get 
\[\delta=2^{n-3}.\]
The proof is complete.
\end{proof}

By Theorem 14 in \cite{Jakli}, $D$ is a Euclidean distance matrix. Since $D$ is invertible, by (P\ref{(P8)}), $D$ has one positive and $m-1$ negative eigenvalues. We now obtain an interlacing property between the eigenvalues of $\L$ and $D$. 
\begin{theorem}
Suppose the eigenvalues of $D$ and $\L$ are arranged as $$\mu_1>0>\mu_2 \geq \cdots \geq \mu_m,$$ 
and
$$\lambda_1 \geq \cdots \geq \lambda_{m-1} > \lambda_m = 0,$$ respectively. Then
\[ 0>-\frac{2}{\lambda_1}\geq\mu_2\geq -\frac{2}{\lambda_2}\geq 
\cdots \geq -\frac{2}{\lambda_{m -1}}\geq \mu_m.\]
\end{theorem}

\begin{proof}
The eigenvalues of $\L^{\dag}$ are 
\[0<\dfrac{1}{\lambda_1} \leq \cdots \leq \dfrac{1}{\lambda_{m-1}} .\] 
Let $Q$ be an orthogonal matrix such that
$$Q'\L^{\dag}Q=\diag\bigg(0,\dfrac{1}{\lambda_1},\dotsc,\dfrac{1}{\lambda_{m-1}}\bigg).$$
By an easy computation, we see that 
$Q' \diag(\L^{\dag})J Q$ has first column non-zero and remaining columns equal to zero. 
Since 
\[D=\diag(\L^{\dag})J_m+ J_m\diag( \L^{\dag}) -2 \L^{\dag}, \]
it follows that 
$\diag\big(-\frac{2}{\lambda_1},\dotsc,-\frac{2}{\lambda_{m-1}}\big)$ is a principal submatrix of $Q'DQ$.
By interlacing theorem, we deduce
\[\mu_1 \geq 0>-\dfrac{2}{\lambda_1} \geq \mu_2 \geq \cdots \geq -\dfrac{2}{\lambda_{m-1}} \geq \mu_m.\]
The proof is complete.
\end{proof}

\section{Acknowledgement}
I would like to thank Prof. R.B. Bapat and Dr. R. Balaji for their help in improving the presentation of the manuscript.
\bibliography{mybibfile}
\end{document}